\documentclass{article}

\usepackage[english]{babel}
\usepackage{graphicx}
\usepackage{float}
\usepackage{url}
\usepackage{amsmath}
\usepackage{amsthm}
\usepackage{amsfonts}
\usepackage{amssymb}

\usepackage{tikz}
\usetikzlibrary{arrows.meta}  
\usetikzlibrary{quotes}
\usepackage{adjustbox}
\usepackage{babel}
\usepackage{bussproofs}
\usepackage{enumitem}
\usepackage{inputenc}

\newcommand{\tab}{\hspace{10mm}}
\newcommand{\co}{{:}}

\newcommand{\bimpl}{\leftrightarrow}

\newcommand{\halt}{{\downarrow}}

\newcommand{\impl}{\rightarrow}

\newcommand{\Tomega}{\RightLabel{$(\omega)$}}

\newtheorem{theorem}{Theorem}[section]
\newtheorem{lemma}[theorem]{Lemma}

\newtheorem{prop}[theorem]{Proposition}
\newtheorem{example}[theorem]{Example}

\theoremstyle{definition}
\newtheorem{definition}[theorem]{Definition}
\theoremstyle{remark}
\newtheorem{remark}[theorem]{Remark}
\newtheorem*{theorem*}{Lemma}
\newtheorem*{lemma*}{Lemma}
\newtheorem*{definition*}{Definition}
\newtheorem*{remark*}{Remark}

\usepackage[citestyle=numeric]{biblatex}
\title{Properties of Selector Proofs}
\author{Elijah Gadsby\footnote{Mathematics PhD Program, Graduate Center CUNY, 365 Fifth Ave., New York City, NY 10016, USA. egadsby@gradcenter.cuny.edu}}

\begin{document}
\maketitle
\begin{abstract}
    A serial property is a suitably enumerated sequence $\{F_n\}$ of formulas and is called selector provable in PA if there is a PA-recursive function $s(x)$ such that $\text{PA}\vdash \forall x (s(x){:}_{\text{PA}} \ulcorner F_x\urcorner)$ where $x{:}_{\text{PA}} y$ is a suitable proof predicate. These notions were introduced by Artemov in his analysis of consistency and the formalisation of metamathematics. These matters aside, the notion is intimately connected with that of relative consistency.

    This paper will give an overview of the mathematical properties of selector proofs. Topics include: the relationship between selector proofs, ordinary provability and the omega rule, the relationship between selector proofs and relative consistency, iterated selector proofs, and the complexity of selectors.
\end{abstract}

\section{Introduction}
A central feature of logic is the notion of a scheme: from schematic axiomatisations of Hilbert-style proof systems to the axiom schemes of induction or replacement in PA or ZFC. It is common to say that PA proves the least number principle or the principle of collection. This raises the question of what it means to prove a scheme and how this can be formalised.

Artemov \cite{ArtemovSelectorProofs} proposes that the appropriate notion is that of a selector proof. A serial property is an (effectively enumerated) sequence $\{F_n\}$ of formulas. A selector proof for this serial property in a theory $T$ consists of a $T$-recursive function $s(x)$, called a selector, such that
\[T\vdash \forall x (s(x)\co_T F_x)\]
where $x\co_T y$ is a proof predicate for $T$. The principal connection to standard provability is that if a scheme $\{F_n\}$ is selector provable in $T$, then each instance is provable. We shall offer a body of counterexamples where the converse fails, suggesting that selector proofs offer new insights about provability.

A special case of interest is where the serial property consists of all instances of a particular formula. Artemov \cite{ArtemovSelectorProofs} shows that the consistency scheme $\{\neg(\overline{n}\co_{\text{PA}}\bot)\}$ is selector provable in PA, so that PA can selector prove its own consistency. Artemov argues that the serial property of consistency is the appropriate arithmetical representation of the consistency of PA and that his selector proof shows that PA does prove its own consistency. We will not adjudicate the matter here. Discussions can be found in \cite{ArtemovConFormula},\cite{ArtemovSelectorProofs},\cite{Ignjatovic}. Our focus will be on the mathematical properties of selector proofs and their utility as a technical tool in proof theory.

Artemov's example is not an isolated one: further useful formulas which are selector provable but whose universal closure is not can be found in section 3. However, there are limits: Kurahashi and Sinclaire independently noted that PA cannot selector prove the consistency of PA+Con(PA) (see \cite{ArtemovSelectorProofs}). The explanation of this phenomenon comes from the fact that a $\Delta_1$ formula $F(x)$ is selector provable in $T$ precisely when $T\vdash $\text{Con}(T)$\impl \forall x F(x)$ (see Prop \ref{ig}). Applied to consistency schemes, this yields a result first noted by Ignjatović (\cite{Ignjatovic}) that $T$ selector proves the consistency of a theory $W$ precisely when $T\vdash $\text{Con}($T$)$\impl $\text{Con}($W$), i.e. precisely when $W$ is relatively consistent with $T$.\footnote{Caveats must be made regarding whether the selector proofs offered by this method are contentual in Artemov's sense. See Remark \ref{contentiality}.} In particular, $T$ cannot selector prove the consistency of any theory containing $T+\text{Con}(T)$.

While this provides a firm upper limit on Artemov's program, it indicates that selector proofs of consistency could be a useful tool for obtaining such relative consistency results. In fact, this method has already been successfully employed by Freund and Pakhomov (\cite{FreundPakhomov}) in their proof-theoretic treatment of slow consistency. They consider the statement $\text{SRefl}_{\Sigma_1}$(PA) of slow $\Sigma_1$ reflection and construct a primitive recursive function $s(x)$ such that
\[I\Sigma_1\vdash \forall (s(x)\co_{\text{PA}}\neg (x\co_{\text{PA}+\text{SRefl}_{\Sigma_1}(\text{PA}))}\bot))\]
and from this conclude that
\[I\Sigma_1\vdash \text{Con(PA)}\impl \text{Con(PA}+\text{SRefl}_{\Sigma_1}(\text{PA}))\]
A consequence of this is that for each $\alpha<\epsilon_0$, PA selector proves the consistency of $\text{PA+SCon}^\alpha$(PA), where $\text{SCon}^\alpha$(PA) is the $\alpha$th iterate of the slow consistency statement. While these results had been earlier established by model-theoretic methods (see \cite{FreundSlowRefl},\cite{HenkPakhomov}),\footnote{The case for finite $n$ was established model theoretically by Rathjen, Friedman, and Weiermann \cite{RFWSlowCon} Sec 4. These cases also follow from a bound on the length of the Gentzen reduction procedure in terms of $F_{\epsilon_0}$.} as far as the author is aware, all proof-theoretic treatments of this subject have involved the method of selector proofs.

From the perspective of Artemov's program, these results also show that despite the limitations provided by the Kurahashi-Sinclaire observation, there are infinitely many proper sound extensions of PA in the language of arithmetic whose consistency is selector provable in PA.

Our goal is to present the beginnings of a general theory of selector proofs. Section 2 contains arithmetical preliminaries while Section 3 presents a useful and representative body of examples of formulas that are selector provable in PA, but not provable. Section 4 discusses the classification of selector provability in terms of extensions of the theory $T$. The aforementioned result regarding $\Delta_1$ formulas is examined, and it is shown that there is no analogous characterisation for more complex classes of formulas. Iterated selector provability is also considered and it is shown that such iterated selectors provide no more power than a single one. Section 5 compares selector provability with verifiable deduction in $\omega$-logic to obtain a precise classification of the formulas selector-provable in PA. While provability in PA corresponds to verifiable $\omega$-logic derivations of height $<\epsilon_0$, selector provability corresponds to such derivations of height $\leq \epsilon_0$. A computational analysis of the recursive functions whose totality can be selector proven in PA is also given. Finally, Section 6 will discuss the complexity of selectors. While elementary or primitive recursive selectors seem to suffice in natural cases, in principle all provably recursive functions of a theory are necessary.

\section{Preliminaries}
The basic setup will consist of two consistent, elementary-recursively axiomatisable theories $S$ and $T$ formulated in the language of arithmetic $\{0,S,+,\cdot,<\}$ such that $S\subseteq T$.\footnote{Analogous results could be obtained when $S$ or $T$ are theories interpreting (a sufficient amount of) arithmetic, but we stick to this case for simplicity.} The base theory $S$ will be assumed to be $\Sigma_1$-sound and to contain $I\Sigma_0+\text{Exp}$ where Exp is a sentence asserting the totality of $\lambda x.2^x$ according to a suitable $\Sigma_0$ definition of its graph (see \cite{HajekPudlak} sec 5.3(c)]).\footnote{Note that no soundness assumptions are being made regarding $T$ beyond its consistency. When such assumptions are required, they will be explicitly noted.} This ensures that it is strong enough to code sets, sequences, syntax, and to define well-behaved partial truth predicates for the language of arithmetic. When concreteness is required, we shall use PA as a representative example.

We fix a suitable coding of sequences and sets and use $[n_0,...,n_{k-1}]$ as shorthand for the code of the sequence of length $k$ whose entries are $n_0$,...,$n_{k-1}$. For sequences $\sigma$ and $\tau$, $(\sigma)_i$ refers to the $i$th element of $\sigma$ and $\sigma*\tau$ the concatenation of $\sigma$ and $\tau$.

\begin{definition}
    Let $\mathcal{F}$ be set of functions on $\mathbb{N}$. The elementary closure $\mathcal{E}(\mathcal{F})$ of $\mathcal{F}$ is the smallest set of functions that contains the basic functions, truncated subtraction, the functions in $\mathcal{F}$, and is closed under composition and bounded sums and products. 
    
    The functions in $\mathcal{E}(\mathcal{F})$ will be called elementary recursive in $\mathcal{F}$ and a relation is elementary recursive in $\mathcal{F}$ if it's characteristic function is. If $\mathcal{F}=\{f\}$, then we shall just write $\mathcal{E}(f)$ and refer to functions being elementary recursive in $f$. The elementary recursive functions are the functions in $\mathcal{E}(\emptyset)$. Note that $\lambda x.2^x\in \mathcal{E}(\emptyset)$
\end{definition}

Relations in $\mathcal{E}(\mathcal{F})$ are closed under Boolean connectives, bounded quantifiers and bounded minimisation. The base theory $S$ is sufficiently strong that all elementary recursive functions are provably recursive, and all elementary recursive relations are provably $\Delta_1$. 

We fix a G\"odel numbering and a standard proof predicate $x\co_Ty$ such that $x\co_Ty$ is provably $\Delta_1$ in $S$ and all of the relevant synctactic operations are $S$-recursive. The corresponding provability predicate for $T$ will be denoted by $\Box_T$. In general, we shall omit G\"odel codes and simply write $\Box A$ in place of $\Box\underline{\ulcorner A\urcorner}$. If we write $\Box A(x)$ for a formula containing a free variable $x$, then we mean that the result of substituting the $x$th numeral in place of the variable $x$ is provable ($\Box \ulcorner A(\dot{x})\urcorner$ in the usual dot notation). The same conventions will apply to $x\co_Ty$ and to partial truth predicates. Since $S$ is sufficiently strong, $\Box_T$ satisfies the Hilbert Bernays conditions (with free variables) and strong $\Sigma_1$ completeness provably in $S$.

Uniform $\Sigma_n$ reflection for $T$ is the sentence
\[\text{Refl}_{\Sigma_n}(T):= \forall x(\Sigma_n(x)\land \Box_T(x)\impl Tr_{\Sigma_n}(x))\]
where $\Sigma_n(x)$ is an elementary recursive formula defining the $\Sigma_n$ formulas and $Tr_{\Sigma_n}$ is a $\Sigma_n$ partial truth predicate. $\text{Refl}_{\Pi_n}(T)$ is defined analogously.

In Section 4, we will employ the notion of $n$-provability. An overview can be found in \cite{BeklemishevReflection}.

\begin{definition}
    For each $n$, the $n$-provability predicate $[n]_T$ for $T$ is defined by
    \begin{equation}
        [n]_T(x) := \exists y (Tr_{\Pi_n}(y)\land \Box_T(Tr_{\Pi_n}(y)\impl x))
    \end{equation}
    The formula $[n]_T(x)$ is thus $\Sigma_{n+1}$.
\end{definition}
The relevant properties of this notion are the following.
\begin{lemma}[\label{nprov}\cite{BeklemishevReflection} Sec 2.3]If $T$ is sound, then the following hold:
\begin{enumerate}
    \item $n$-provability for $T$ satisfies the Hilbert Bernay's Conditions with free variables (in $S$)
    \item (Strong $\Sigma_{n+1}$ completeness) For each $\Sigma_{n+1}$ formula $F(\vec{x})$, $S\vdash F(\vec{x})\impl [n]_TF(\vec{x})$
    \item $S\vdash [0]_T(x)\bimpl \Box_T(x)$
    \item $S\vdash \neg [n]_T\bot \bimpl \text{Refl}_{\Pi_{n+1}}(T)$
\end{enumerate}
\end{lemma}

We close the preliminaries with a brief review of the provably recursive functions of PA. (For more information, see \cite{Arai},\cite{Wainer},\cite{Sommer}.)  Let $\prec$ be an elementary recursive well-order of order type $\epsilon_0+1$. We shall use Greek letters for variables ranging over ordinals $\preceq \epsilon_0$, i.e. elements of the field of $\prec$.  

The Wainer Hierarchy is the Hierarchy of fast growing functions defined by transfinite recursion with
\begin{flalign*} 
    &F_0(x) = x+1 \\
    &F_{\alpha+1}(x) = F_\alpha^{x+1}(x) \\
    &F_\lambda(x) = F_{\lambda[x]}(x) \text{ for $\lambda$ a limit ordinal}
\end{flalign*}
where $\lambda[x]$ is a fundamental sequence for $\lambda$ and $g^m$ is the $m$th iterate of $g$.

The following result by Sommer shows that we can talk about such functions in arithmetic.

\begin{lemma}[\cite{Sommer} Sec 5.2]\label{Sommer}
    There is a $\Sigma_0$ formula $F_\alpha(x)=y$ with free variables $\alpha,x,y$ that uniformly defines the graphs of $F_\alpha$ for $\alpha<\epsilon_0$. Furthermore, $I\Sigma_0+\text{Exp}$ proves that for each $\alpha$, this formula defines a monotone partial function.
\end{lemma}
$F_{\epsilon_0}$ can then be defined directly by $F_{\epsilon_0} (x)=F_{\omega_{x+1}} (x)$ where $\omega_0(\alpha)=\alpha$, $\omega_{n+1}(\alpha)=\omega^{\omega_n(\alpha)}$ and $\omega_n=\omega_n(1)$.

Let $\alpha\impl_x\beta$ mean that there is a sequence $\alpha=\alpha_0,\alpha_1,...,\alpha_k=\beta$ such that $\alpha_{i+1}=\alpha_i[k]$ and let $|\alpha|$ be the number of $\omega$s in the complete Cantor normal form of $\alpha$ (where coefficients are not used). The following fact will be needed.

\begin{lemma}[\cite{Arai} 4.10, 4.12]\label{redprop} \textcolor{white}.
    \begin{enumerate}
        \item If $\omega^\alpha\impl_k\omega^\beta$, then $F_\beta(k)\geq F_\alpha(k)$
        \item If $\alpha\to_k\beta$, and $k\leq l$, then $\alpha\to_l\beta$
        \item If $\beta<\alpha$ then $\alpha\to_{|\beta|}\beta$
    \end{enumerate}
    
\end{lemma}

The following is a classic result of the proof-theory.
\begin{theorem}[\cite{Wainer} 4.12] \label{class}\textcolor{white}{.}

\begin{enumerate}
\item If $2\leq\alpha<\beta\leq\epsilon_0$ then $\mathcal{E}(F_\alpha)\subsetneq \mathcal{E}(F_\beta)$. In particular, $F_\beta$ dominates every function in $\mathcal{E}(F_\alpha)$.
\item The provably recursive functions of PA are precisely those in $\cup_{\alpha<\epsilon_0}\mathcal{E}(F_\alpha)$
\end{enumerate}
\end{theorem}

In section 5, we will be interested in the complexity of selectors for serial properties selector provable in PA. This information could be obtained by a thorough analysis of the arguments, but it will be more convenient to simply work in a theory for which the provably recursive functions are understood. For a definable partial function $f(\vec{x})$, we write $f(\vec{x})\halt$ for $\exists y(f(\vec{x})=y)$ and $f\halt$ for $\forall \vec{x}\exists y (f(\vec{x})=y)$ where $f(\vec{x})=y$ is the formula defining its graph. The following result will be useful.

\begin{prop}[\cite{Arai} 2.13]\label{er}
    Let $f(\vec{x})$ be a monotone function with a $\Sigma_0$ definable graph such that $I\Sigma_0+\text{Exp}$ proves that this graph defines a partial function.
    \begin{enumerate}
        \item Every function that is provably recursive in $I\Sigma_0+\text{Exp}+f\halt$ is elementary recursive in $f$. 
        \item If $I\Sigma_0+\text{Exp}+f\halt$ proves that $f$ is monotone, then every function in $\mathcal{E}(f)$ is provably recursive in $I\Sigma_0+\text{Exp}+f\halt$.
    \end{enumerate}
\end{prop}

\section{Selector-Provable Properties and Formulas}
In this section, we formally introduce selector provability, provide some useful and representative examples of selector provable formulas and schemes. None of the examples are new.

\begin{definition}
    A serial property $\{F_n\}$ is an $S$-recursively enumerated sequence $\{F_n\}$ of formulas. We say that the serial property $\{F_n\}$ is selector provable in $T$ (over the base theory $S$) if there is an $S$-recursive function $s(x)$ such that
    \[S\vdash \forall x (s(x)\co_T F_x)\]
\end{definition}

We emphasise that while the choice of (standard) proof predicate is unimportant, all of our results are all relative to a fixed proof predicate for $T$. For results when the proof predicate is permitted to vary, see \cite{SantosSiegKahle}.

We generally omit reference to $S$ and refer to such formulas as selector provable in $T$. The definition extends naturally to multi-indexed schemes. Of particular interest is the case where the serial property consists of all numerical instances of a particular formula. 

\begin{definition}
    Let $F(\vec{x})$ be a formula. We say $F(\vec{x})$ is selector provable in $T$ (over $S$) if there is an $S$-recursive function $s(\vec{x})$ such that
    \[S\vdash \forall x(s(\vec{x})\co_T F(\vec{x}))\]
\end{definition}

\begin{remark}
    Call a formula $F(\vec{x})$ implicitly selector provable in $T$ (over $S$) if $S\vdash \forall \vec{x} \Box_TF(\vec{x})$. Since $S$ proves the least number principle for the $\Delta_1$ formula $y\co_TF(\vec{x})$ it follows that if $F(\vec{x})$ is implicitly selector provable, then it is selector provable. (The $\Sigma_1$ soundness of $S$ is necessary here to ensure that the result is a genuine recursive function.) However, the notions remain distinct. For any $T$-recursive function $r(\vec{x})$, one has $T\vdash \forall \vec{x}(r(\vec{x})\co_TF(\vec{x}))\impl F(\underline{\vec{n}})$ for each $\vec{n}$ assuming only the consistency of $T$.\footnote{Since $r$ is a genuine recursive function, there is $m$ such that $T\vdash r(\vec{n})=m$. The result then follows from the fact that $T\vdash \underline{m}\co_TF(\underline{\vec{n}})\impl F(\underline{\vec{n}})$.} The notions also diverge sharply when iterated selectors are considered (see Proposition \ref{iterated} and Remark \ref{hirohiko}).

    For simplicity of notation and argumentation, we shall generally work with implicit selector proofs. For negative results and proof-theoretic purposes, they are entirely adequate. The constructions used in positive results could all be carried out entirely using explicit selectors if desired.\footnote{In particular, all uses of strong $\Sigma_1$ completeness are applied to formulas $\Delta_1$ in $S$ and so could be replaced with explicit $S$-recursive verifying functions.}
\end{remark}

We begin some useful and representative examples of formulas that are selector provable (but not provable). None of the examples are new, though they serve as a nice illustration of the concept and will be employed in later sections.

The first example is Artemov's result.
\begin{example}[Artemov \cite{ArtemovSelectorProofs}]
Let $T$ be formulated in the language of arithmetic.\footnote{The argument applies more generally to any elementary-recursively axiomatised $T$ which can define well-behaved partial truth predicates.} Then $T$ selector proves its own consistency over $I\Sigma_0+\text{Exp}$, i.e. there is an elementary recursive function $s(x)$ such that
\[I\Sigma_0+\text{Exp}\vdash \forall x (s(x)\co_T\neg(x\co_T\bot)) \]
\end{example}

Next, we have the principle of explicit reflection $\underline{n}\co_T A\impl A$ for a sentence $A$. It was first notice by G\"odel (\cite{GodelWorks}) that, unlike the implicit reflection principle $\Box_T A\impl A$, all instances of explicit reflection for $T$ are provable in $T$. 

\begin{example}\label{lemma}[Feferman \cite{FefermanTransfiniteRP}, see also \cite{BeklemishevReflection} 2.2]
    Let $F_n$ be an elementary recursive enumeration of all sentences. Then
    \begin{equation}
        S\vdash \forall xy\Box_T(x\co_T F_y \impl F_y)
    \end{equation}
    In particular, for each formula $F(\vec{x})$
    \begin{equation}\label{pointwise}
        S\vdash \forall \vec{x}y\Box_T(y\co_T F(\vec{x}) \impl F(\vec{x}))
    \end{equation}
\end{example}

The next example concerns transfinite induction. We confine ourselves to PA. For a formula $A(x)$ with free variable $x$, 
\[\text{Prg}(A) := \forall \alpha((\forall\beta\prec\alpha) A(\beta)\impl A(\alpha))\]
and
\[\text{TI}(\alpha,A):= \text{Prg}(A)\impl (\forall\beta\prec\alpha)A(\beta)\]

\begin{example}[\cite{Arai} 4.5, \cite{FreundPakhomov}]\label{TI}
For each formula $A(x,\vec{y})$
\[I\Sigma_0+\text{Exp}\vdash (\forall\alpha\prec\epsilon_0)\forall\vec{y}\,\Box_{\text{PA}}\text{TI}(\alpha,A(x,\vec{y}))\]
\end{example}

\begin{proof}
    A classical result of Gentzen is that $\text{PA}\vdash \text{TI}(\underline{\alpha},A(x,\vec{y}))$ for each $\alpha<\epsilon_0$. Formalising this argument yields the result.
\end{proof}

\begin{example}\label{Fhalt}\cite{FreundPakhomov}
    $I\Sigma_0+\text{Exp}\vdash \forall x \Box_{\text{PA}}F_{\epsilon_0}(x)\halt$
\end{example}
\begin{proof}
    Following \cite{Sommer}, a computation of $F_{\epsilon_0}(n)$ is coded as an elementary-recursive strictly decreasing sequence of ordinals below $\omega_{n+2}$. The sequence continues until it reaches some $m<\omega$ with said $m$ the value of $F_{\epsilon_0}(n)$. That such a sequence terminates follows from a suitable instance of transfinite induction so the result follows from Example \ref{TI}.
\end{proof}

\begin{remark}\label{complowerbound}
    It follows that $\text{PA}\vdash\forall x\Box_{\text{PA}} F_{\epsilon_0}(g(x))\halt$ for any PA-recursive function $g$, in particular, $\text{PA}\vdash \forall x \Box_{\text{PA}} F_{\epsilon_0}(F_{\underline{\alpha}}(x))\halt$ for each $\alpha<\epsilon_0$. Proposition \ref{companalysis} shows that this result is optimal.
\end{remark}

Further examples of formulas that are selector provable in PA can be found by consulting the traditional independence results. Details of the statements can be found in \cite{BeklemishevReflection},\cite{HajekPudlak},\cite{ParisKirby}.

\begin{prop}
    The following are selector provable in PA:
    \begin{enumerate}
        \item The totality of the Goodstein function (i.e. every special Goodstein sequence terminates)
        \item Hercules always wins the Paris-Kirby hydra battle 
        \item Beklemishev's worm principle
        \item The Paris Harrington Principle
    \end{enumerate}
    For 1-3, $I\Sigma_0+\text{Exp}$ suffices as a base theory, while for 4, we use $I\Sigma_1$ for simplicity.
\end{prop}
\begin{proof}
    For 1 and 2, the statements can be proven using transfinite induction up to $\epsilon_0$. However, each instance only requires transfinite induction up to a particular ordinal $<\epsilon_0$, whence the results follow from Proposition \ref{TI}. 3 was established directly by Beklemishev (see \cite{BeklemishevReflection} Lemma 5.21).

    For 4, one can take a detour through $\text{ACA}_0$. Let $RT(x)$ be the infinite Ramsey theorem in dimension $x$ (see \cite{Simpson}). Although $\forall x RT(x)$ is not provable in $\text{ACA}_0$ Simpson \cite{Simpson} shows that
    \[\text{ACA}_0\vdash RT(0)\]
    and
    \[\text{ACA}_0\vdash RT(x)\impl RT(x+1)\]
    Imitating the standard proof of the Paris Harrington Principle from the infinite Ramsey Theorem, one can then show
    \[\text{ACA}_0\vdash RT(x)\impl \forall q PH(q,x)\]
    Putting this together yields
    \[I\Sigma_1\vdash \forall x \Box_{\text{ACA}_0} \forall q PH(q,x)\]
    and so 
    \[I\Sigma_1\vdash \forall x \Box_{\text{PA}}\forall q PH(q,x)\]
    by formalising the conservativity of $\text{ACA}_0$ over PA.
\end{proof}

\section{Selector and Ordinary Provability}
Here we investigate the relationship between the selector provability of a formula in theory $T$ and its provability in extensions of $T$. From this, examples are obtained of formulas that are instance provable, but not selector provable, in given theories.

Our starting point is the classification of the selector-provable $\Delta_1$ formulas mentioned in the introduction.

\begin{prop} \label{ig}
Let $F(\vec{x})$ be a formula provably $\Delta_1$ in $S$. Then
\[S\vdash \forall x \Box_TF(\vec{x}) \bimpl (\text{Con}(T)\impl \forall \vec{x} F(\vec{x}))\]
\end{prop}
\begin{proof}
The left-to-right direction follows from the fact that $\text{Con}(T)$ is equivalent to $\text{Refl}_{\Pi_1}(T)$ over $S$. From right to left, working in $S$, assume that $\text{Con}(T)\impl \forall \vec{x} F(\vec{x})$ and fix $\vec{x}$. If $F(\vec{x})$ holds, then $\Box_T F(\vec{x})$ by strong $\Sigma_1$ completeness, whereas if $\neg F(\vec{x})$, then $\Box\bot$ by the assumption whence $\Box F(\vec{x})$ by formalised explosion.
\end{proof}

As noted in the introduction, applying this to the formula $\neg x\co_W\bot$ for a theory $W$ yields Ignjatović's result (\cite{Ignjatovic}) 
\[S\vdash \forall x \Box_T \neg (x\co_W\bot) \bimpl(\text{Con}(T)\impl \text{Con}(W))\]
In other words, $T$ selector proves the consistency of $W$ over $S$ precisely when $W$ is relatively consistent with $T$ over $S$, in particular, $T$ cannot selector prove the consistency of any theory containing $T+\text{Con}(T)$.

\begin{remark}\label{contentiality}
    It should be noted that the selector proofs of consistency provided by Proposition \ref{ig} are non-contentual in Artemov's sense (\cite{ArtemovSelectorProofs} Section 5.2). In practice, however, reductions from $W$ to $T$ given by proof-theoretic methods tend to involve effective procedures transforming proofs of a contradiction in $W$ to ones in $T$. Combining this with a contentual selector consistency proof for $T$ in $T$ yields the desired contentual selector consistency proof for $W$.
\end{remark}

From this result we can derive further examples of formulas that are not selector provable, namely generalised reflection principles. Let $\Box_T^0A=A$ and $\Box_T^{n+1}A=\Box_T\Box_T^n A$. Nogina (\cite{Nogina} 2) shows that if $T$ is $\Sigma_1$-sound, then $T\vdash \underline{n}\co_T\Box^kA\impl A$ for each $n,k$ and sentence $A$. However, Unlike the standard explicit reflection statements, these are not selector provable.\footnote{Nogina's argument relies on the $\Sigma_1$ soundness of $T$ to eliminate provability predicates and as such cannot be formalised in $T$.}

\begin{example} For each $k\geq 1$
    \begin{equation}\label{noginalemma}
        S\not\vdash \forall y\Box_T(y\co_T\Box^k\bot\impl \bot)
    \end{equation}
\end{example}
\begin{proof}
    Suppose this were the case. If $k\geq 2$, then since $S\vdash \Box_T\bot\impl \Box_T^k\bot$, there is an elementary recursive function $s(y)$ such that $S\vdash y\co_T\Box\bot\impl s(y)\co_T\Box_T^k\bot$. But then $S\vdash \Box_T(y\co_T\Box_T\bot\impl s(y)\co_T\Box_T^k\bot)$ which combined with (\ref{noginalemma}) yields $S\vdash\forall y\Box_T(y\co_T\Box_T\bot\impl\bot)$.
    
    We may thus assume $k=1$. Combining (4) with a formalisation of the deduction theorem yields $S\vdash \forall x\Box_T\neg (x\co_{T+\text{Con}(T)}\bot)$ which contradicts Proposition \ref{ig}.  
\end{proof}

If a formula is selector provable in $T$ (over $S$), then said formula is provable outright in $S$ augmented with a suitable uniform reflection principle for $T$. Thus, further examples of formulas that are instance but not selector provable in $PA$ can obtained by consulting $\Pi_2$ independence results for stronger theories, for example the Buchholz Hydra Game (see \cite{BuchholzHydra}) or Friedman's miniaturisations of Kruskal's Theorem (see \cite{FriedmanKruskal}).

One may hope that uniform reflection principles would allow for a generalisation of Proposition \ref{ig} to more complex formulas. Unfortunately, the right-to-left direction of (\ref{maineq}) fails. Indeed, no analogous result is available as soon as $\Sigma_1$ or $\Pi_1$ formulas are admitted.

\begin{prop}\label{main}
    Let $\Gamma$ be a set of formulas containing every $\Sigma_1$ formula or every $\Pi_1$ formula. There is no set of sentences $\Delta$ such that for each formula $F(\vec{x})$ in $\Gamma$
    \begin{equation}\label{maineq}
        S\vdash \forall \vec{x} \Box_T F(\vec{x}) \tab \text{iff} \tab S+\Delta\vdash \forall \vec{x} F(\vec{x})
    \end{equation}
\end{prop}
Note that no assumptions are made on $\Delta$. It may be infinite and need not be recursive (or even arithmetical).

\begin{proof}
Suppose that $\Delta$ were a set of sentences satisfying (\ref{maineq}). 

First, consider the case where $\Gamma$ contains every $\Pi_1$ formula. By Proposition \ref{ig} and (\ref{maineq}), $S+\Delta\vdash \text{Con}(T)$. Since $\text{Con}(T)$ is a $\Pi_1$ formula and so (\ref{maineq}) yields $S\vdash \forall\vec{x}\Box_T\text{Con}(T)$. But then $T\vdash \text{Con}(T)$ by the $\Sigma_1$ soundness of $S$ which is not the case.

Thus, suppose that $\Gamma$ contains every $\Sigma_1$ formula.
Example \ref{lemma} applied the $\Sigma_1$ formula $Tr_{\Sigma_1}(x)$ yields
\[S\vdash \forall xy \Box_T (y\co_T Tr_{\Sigma_1}(x)\impl Tr_{\Sigma_1}(x))\]
The formula $y\co_T Tr_{\Sigma_1}(x)\impl Tr_{\Sigma_1}(x)$ is provably $\Sigma_1$ in $S$ and so (\ref{maineq}) yields
\[S+\Delta \vdash \forall xy( y\co Tr_{\Sigma_1}(x)\impl Tr_{\Sigma_1}(x))\]
whence $S+\Delta \vdash \text{Refl}_{\Sigma_1}(T)$. But then $S+\Delta \vdash \text{Con}(T+\text{Con}(T))$ and (\ref{maineq}) yields $S\vdash \forall x \Box_T\neg(x\co_{T+\text{Con}(T)}\bot)$, contradicting Proposition \ref{ig}.
\end{proof}

While this approach to classifying the selector provable formulas does not succeed, in the next section we give a precise characterisation using verifiable derivations in $\omega$-logic. In the meantime, we note that an analogous result to Proposition \ref{ig} can be obtained for more complex formulas using $n$-provability.

\begin{prop}
    Suppose that $T$ is sound. If $F(\vec{x})$ is provably $\Delta_{n+1}$ in $S$, then
\begin{equation}
    S\vdash \forall \vec{x} [n]_TF(\vec{x}) \bimpl (\text{Refl}_{\Pi_{n+1}}(T)\impl \forall \vec{x} F(\vec{x}))
\end{equation}
The $n=0$ case is just Proposition \ref{ig}.
\end{prop}
\begin{proof}
    Working in $S$, in one direction assume $\forall \vec{x}[n]_TF(\vec{x})$ and $\text{Refl}_{\Pi_{n+1}}(T)$, and fix $\vec{x}$. Then there is $y$ such that $Tr_{\Pi_n}(y)$ and $\Box_T(Tr_{\Pi_n}(y)\impl F(\vec{x}))$. $\text{Refl}_{\Pi_{n+1}}(T)$ then yields $Tr_{\Pi_n}(y)\impl F(\vec{x})$ and thus $F(\vec{x})$ as required.

    In the other direction, working in $S$ assume $\text{Refl}_{\Pi_{n+1}}(T)\impl \forall \vec{x} F(\vec{x})$ and fix $\vec{x}$. If $F(\vec{x})$, then $[n]_TF(\vec{x})$ by \ref{nprov}.2. Meanwhile, if $\neg F(\vec{x})$, then the assumption yields $\neg \text{Refl}_{\Pi_{n+1}}(T)$. So $[n]_T\bot$ by \ref{nprov}.4 and $[n]_TF(x)$ by formalised explosion.
\end{proof}

We now consider iterated selector proofs. Since one additional layer of coding allows for more formulas to become (selector) provable, one may hope that additional layers of coding would allow additional formulas to be proven and thus overcome the limit to Artemov's program.

\begin{definition}
    A formula $F(x)$ is called $2$-selector provable in $T$ if there are $S$-recursve functions $r(x)$,$s(x)$ such that
    \[S\vdash \forall x (s(x)\co_T(r(x)\co_TF(x)))\]
\end{definition}
This could naturally be extended to $n$-selector provability. Unfortunately, iterated selectors do not provide additional power beyond that of a single one. 

\begin{prop}\label{iterated}
    If a formula $F(x)$ is $2$-selector provable in $T$, then it is selector provable in $T$.
\end{prop}
\begin{proof}
    Let $r(x)$ and $s(x)$ be $S$-recursive functions. We show that
    \begin{equation}
        S\vdash s(x)\co_T(r(x)\co_TF(x))\impl \Box_TF(x)
    \end{equation}
    from which the result follows.
    
    Working in $S$, assume $s(x)\co_T(r(x)\co_TF(x))$, so $\Box_Tr(x)\co_TF(x)$. By Lemma \ref{lemma} we have
    \[\forall y \Box(y\co_T\co F(x)\impl F(x))\]
    So $\Box_T F(x)$ follows by instantiating $y$ with $r(x)$ and formalised modus ponens.
\end{proof}

\begin{remark}\label{hirohiko}
This result should be contrasted with that of Kushida (\cite{Kushida} 3.1) who shows that if $T$ is $\Sigma_1$ sound, then for each $n$, there is a formula $F_n(x)$ such that $S\vdash \forall x \Box_T^nF_n(x)$, but $S\not\vdash\forall x \Box^k_TF_n(x)$ for any $k<n$. Although for most purposes they are equivalent, explicit selectors and provability predicates differ sharply in the iterated case.
\end{remark}

As final matter, we consider the conditions under which selectors can be removed from a proof. Given that we have proven $S\vdash \forall \vec{x} \Box_TF(\vec{x})$, under what conditions we conclude that in fact $T\vdash \forall \vec{x}F(\vec{x})$? For theories such as PA that contain full induction in their language, we may use the following criterion.

\begin{definition}
    A selector provable formula $F(\vec{x})$ is called compact (over $T$) if there is a finite fragment $W\subseteq T$ and an $S$-recursive function $s(x)$ such that $S\vdash \forall \vec{x}(s(\vec{x})\co_WF(\vec{x}))$. 
\end{definition}

\begin{prop}
    Let $T$ be a finite extension of PA in the language of arithmetic and suppose and $F(\vec{x})$ is selector provable in $T$. Then $F(\vec{x})$ is compact if and only if $T\vdash \forall \vec{x} F(\vec{x})$.\footnote{The same result applies to ZFC and finite extensions thereof.}
\end{prop}
It follows that none of the examples in Section 3 are compact. In each case, the selector requires unbounded access to the induction scheme.

\begin{proof}
    The direction from left to right follows from the fact that T proves uniform reflection principles for each of its finite fragments (see \cite{Leivant}). From right to left, if $T\vdash \forall \vec{x} F(\vec{x})$, then $W\vdash \forall \vec{x}F(x)$ for some finite $W\subseteq T$. So $S\vdash \Box_W\forall x F(x)$, whence $S\vdash \forall x \Box_WF(x)$ by formalised instantiation.
\end{proof}

\section{Selector Provability and the $\omega$-Rule}
In this section, we perform a proof-theoretic analysis of selector provability and obtain a precise characterisation of the selector provable formulas of PA in terms of verifiable deductions in $\omega$-logic along with a bound on the growth rate of recursive functions whose totality is selector provable in PA. 

In this section, we shall assume that PA has been formulated in the Tait calculus with constant symbol $0$, a function symbol for the successor and relation symbols for all elementary recursive predicates (see \cite{Arai},\cite{Buchholz}). In particular, addition and multiplication are treated as relations. In $\omega$-logic, one replaces the standard rule for the universal quantifier with the infinitary rule

\begin{prooftree}
    \AxiomC{$\Gamma,\forall x A(x),A(\underline{0})$}
        \AxiomC{$\Gamma,\forall x A(x), A(\underline{1})$}
            \AxiomC{\dots}
    \Tomega
    \TrinaryInfC{$\Gamma,\forall x A(x)$}
\end{prooftree}
Deductions thus become well-founded labelled trees, and free variables become unnecessary. These can be formalised in PA using the indices of partial recursive functions with well-foundedness being approximated by directly assigning ordinals to the nodes of the tree. 

Let $d\co_{\omega}^\alpha A$ denote that $d$ is the code for an $\omega$-logic derivation of the sentence $A$ with height $\leq \alpha$ (see definition \ref{treedef}). The following classification of the PA-provable sentences is known.

\begin{theorem}[\cite{Sundholm} 5(f)]
    A sentence $A$ is provable in PA if and only if $\text{PA}\vdash \underline{e}\co_\omega^{\underline{\alpha}} A$ for some $e$ and $\alpha<\epsilon_0$.
\end{theorem}

The analogous characterisation for selector provability is the following.
\begin{theorem}\label{omega}
    Let $A(x)$ be a formula with one free variable. Then $\text{PA}\vdash\forall x\Box_{\text{PA}}A(x)$ if and only if $\text{PA}\vdash \underline{e}\co_\omega^{\epsilon_0}\forall x A(x)$ for some $e$.
\end{theorem}

Let $\{e\}(x)$ be a universal partial recursive function (of one variable). For index $e$, we define $supp(e):= \{n\;|\; \{e\}(n)\halt\land \{e\}(n)\neq 0\}$.

For our formal definition of $\omega$-logic derivations, we will use labelled sequents of the form $<r,\Gamma, A,\alpha>$, with $r$ (the label of) a rule, $\Gamma$ a finite set of formulas, $A$ a formula in $\Gamma$ (the principal formula of the inference) or $0$, and $1\preceq\alpha\preceq \epsilon_0$.\footnote{We assume that the coding is such that $0$ is neither a formula nor the label of a sequent.} For a finite sequence $\sigma$, we use $rule(\sigma)$, $seq(\sigma)$, $PF(\sigma)$, and $o(\sigma)$ as shorthand for $\pi_0(\{e\}(x))$, $\pi_1(\{e\}(x))$, $\pi_2(\{e\}(x))$, and $\pi_3(\{e\}(x))$ respectively with $e$ implicit.

\begin{definition}\label{treedef}
    The relation $e\co_\omega^\alpha A$ stating that $e$ is a proof tree with height $\preceq \alpha$ holds if
\begin{enumerate}
    \item $e$ is total and for all $x$, $\{e\}(x)$ is $0$ or a sequent label
    \item $supp(e)$ is a tree
    \item $seq([]) = \{A\}$ and $o([])\preceq \alpha$
    \item $e$ is locally correct, i.e. for any $\sigma\in supp(e)$
        \begin{enumerate}
            \item If $\sigma$ is a leaf node (i.e. $\forall x(\sigma*[x]\not\in supp(e)$) then $rule(\sigma)=Ax$, $PF(\sigma)=0$, and $seq(\sigma)$ contains a true atomic sentence
            \item If $r(\sigma)=\land$, then there are sentences $A,B$ such that $PF(\sigma)=A\land B$, $\sigma*[0],\sigma*[1]\in supp(e)$, $seq(\sigma*[0])= seq(\sigma)\cup\{A\}$, and $seq(\sigma*[1])= seq(\sigma)\cup\{B\}$
            \item If $r(\sigma)=\lor$, then there are sentences $A,B$ such that $PF(\sigma)$ is $A\lor B$ or $B\lor A$, $\sigma*[0]\in supp(e)$, and $seq(\sigma*[0])= seq(\sigma)\cup\{A\}$
            \item If $r(\sigma)= \exists$, then there is $n$, $A$, and a variable $x$ such that $PF(\sigma) = \exists x A$, $\sigma*[0]\in supp(e)$, and $seq(\sigma*[0])= seq(\sigma)\cup A(x/\underline{n})$
            \item If $r(\sigma)=\omega$, then there is a formula $A$ and a variable $x$ such that $PF(\sigma)=\forall x A$ and for each $n$, $\sigma*[n]\in supp(e)$ with $seq(\sigma*[n])=seq(\sigma)\cup \{A(x/\underline{n})\}$
            \item If $r(\sigma)=Rep$, then $PF(\sigma)=0$, $\sigma*[0]\in supp(e)$, and $seq(\sigma*[n])=seq(\sigma)$
        \end{enumerate}
    \item If $\sigma,\tau \in supp(e)$ and $\sigma\subsetneq \tau$, then $o(\tau)\prec o(\sigma)$
\end{enumerate}
\end{definition}

\begin{proof}[Proof of Theorem \ref{omega}]
    For the right-to-left direction, suppose $\text{PA}\vdash(\underline{e}\co_\omega^{\epsilon_0}\forall x A(x))$ with $A$ having degree $q$. By an induction on the length of sequences, PA proves that every formula occurring in the proof is a sentence of degree $\leq q+1$. Let $Tr_{q+1}$ be a partial truth predicate for sentences of degree $\leq q+1$ satisfying the Tarski conditions and define the formula 
    \[T(\alpha):= (\forall \sigma\in supp(\underline{e}))(o(\sigma)=\alpha\impl (\exists B\in seq(\sigma))Tr_{\Pi_{q+1}}(B)))\]

    It is straightforward to see that $\text{PA}\vdash \text{Prg}(T(\alpha)$: the local correctness conditions for $\omega$-logic are essentially the truth conditions for sentences. It follows from Proposition \ref{TI} that $\text{PA}\vdash \forall \alpha\prec\epsilon_0\Box_{\text{PA}}\text{TI}(\alpha,T)$.
    
    Working in PA, $r([])$ is $\omega$ or $Rep$. Suppose that it is $\omega$. Fixing $n$, we have $[n]\in supp(\underline{e})$, $seq([n])=\{\forall x A(x), A(\underline{n})\}$, and $o([n])\prec\epsilon_0$. By strong $\Sigma_1$ completeness these facts are provable in PA, so $\Box_{\text{PA}}(\exists B\in seq([n]))Tr_{q+1}(B)$, $\Box_{\text{PA}}(Tr_{q+1}(\forall x A(x))\lor Tr_{q+1}( A(\underline{n})))$, and $\Box_{\text{PA}}A(n)$ as required. The $r([])=Rep$ case is similar.

    The left-to-right direction amounts to a formalisation of the embedding of PA into $\omega$-logic with cut and of the cut elimination process. Given this, one can stitch together the proof trees for each $n$ and apply the $\omega$ rule to get the required proof of $\forall x A(x)$. For the formalisation, we employ the notation system of Buchholz \cite{Buchholz}.\footnote{Alternatively, one can directly formalise the embedding and cut elimination process through effective operators on the indices of partial recursive functions. See \cite{MintsCtsCE},\cite{Sundholm}.} In the notation system, if $d$ is a PA derivation of a sentence $B$, then $E^kd$ where $k$ is the cut degree of $d$ is a notation for a cut-free $\omega$-logic proof of $B$ with height $\prec\epsilon_0$. Furthermore, there are primitive recursive functions $\mathfrak{r}(h)$, $\mathfrak{e}(h)$ $\mathfrak{s}_n(h)$, $\mathfrak{o}(h)$ returning the final rule, a notation for the $n$-th subderivation, endsequent and ordinal assigned to the notation $h$. It can be proven in $I\Sigma_1$ that these functions satisfy local and ordinal correctness conditions (see \cite{Buchholz} 3.8). 

    Suppose that $\text{PA}\vdash \forall x \Box_{\text{PA}}A(x)$. Then there is a PA-recursive function $n\mapsto d_n$ such that $d_n$ is a PA proof of $A(\underline{n})$. One can then construct an $\omega$-logic derivation of height $\leq \epsilon_0$ by taking $f([])=<\omega, \{\forall x A(x)\}, \forall x A(x), \epsilon_0>$ and can constructing the $n$th subderivation of $\forall x A(x),A(\underline{n})$ recursively from $E^{k_n}d_n$ where $k_n$ is the cut degree of $d_n$ using the functions of the notation system. Take $e$ such that $\text{PA}\vdash \forall x (\{\underline{e}\}(x)=f(x))$. Properties (1)-(3) will follow directly from the construction while (4) and (5) can be verified using the corresponding properties in \cite{Buchholz} 3.8. 
\end{proof}

\begin{remark}
    From (the proof of) Theorem \ref{omega}, it can be seen that if PA selector proves a formula $A(x)$ and the proofs are such that the height of resulting cut free $\omega$-logic derivations of $A(\underline{n})$ are uniformly bounded by some $\alpha<\epsilon_0$ (verifiably in PA), then $\text{PA}\vdash \forall x A(x)$. However, for this to occur there must be a uniform bound on the cut rank of the derivations of each $A(\underline{n})$, i.e. these derivations must all occur in some finite fragment of PA, so this amounts to the compactness condition from earlier.
\end{remark}

Finally, we provide an analysis of the recursive functions whose totality can be selector proven in PA. From Remark \ref{complowerbound}, $\text{PA}\vdash \forall x \Box_{\text{PA}}F_{\epsilon_0}(F_{\underline{\alpha}}(x))$ for each $\alpha<\epsilon_0$. This is optimal.

\begin{theorem}\label{companalysis}
    If $f$ is a recursive function such that $\text{PA}\vdash \forall \vec{x}\Box_{\text{PA}} f(\vec{x})\halt$, then $f$ is dominated by $F_{\epsilon_0}(F_{\alpha}(max\{\vec{x}\}))$ for some $\alpha<\epsilon_0$.
\end{theorem}
Thus, PA does not selector prove the totality of $F_{\epsilon_0}^2(x)$ or of $F_{\alpha}$ for any $\alpha>\epsilon_0$.

To establish this result, we employ the infinitary proof system found in \cite{Wainer}. This system involves judgements of the form $n:N\vdash^\alpha A$ where $N$ is intended as a finite approximation to the natural numbers used to bound existential quantifiers in $\Sigma_1$ sentences. The relevant facts are as follows:

\begin{lemma}(\cite{Wainer} 2.23, 3.11, 3.15, 4.7)\label{complemma}
\begin{enumerate}
    \item If $h$ is a PA proof of $A$ with depth $d$ and cutrank $r$, then $m:N\vdash^{\omega \cdot d}_rA$ where $m$ is the maximum of all numerical parameters occuring in $A$
    \item If $m:N\vdash^\alpha_rA$, then $m:N\vdash^{\omega_r(\alpha)}_0A$
    \item If $E$ is an elementary recursive relation symbol and $ max\{\vec{m}\}\vdash^\alpha_0 \exists \vec{y}E(\vec{y},\vec{m})$, then $\mathbb{N}\models (\exists \vec{y}\leq F_{\alpha}(max\{\vec{m}\}))E(\vec{y},\vec{m})$
\end{enumerate}
\end{lemma}

We may now prove the result.

\begin{proof}[Proof of Theorem \ref{companalysis} ]
Suppose that PA selector proves the totality of $f$. Then there is an elementary recursive relation symbol $E$ and $\vec{m}$ such that $\exists \vec{z}E(\vec{x},y,\vec{z},\underline{\vec{m}})$ defines the graph of $f$ and $\text{PA}\vdash \forall \vec{x} \Box_{\text{PA}}\exists y\vec{z}E(x,y,\vec{z},\underline{\vec{m}})$. Let $\vec{n}\mapsto h_{\vec{n}}$ be a PA recursive function mapping $\vec{n}$ to a proof $h_{\vec{n}}$ of $\exists y\vec{z}E(\underline{\vec{n}},y,\vec{z},\vec{m})$, let $d(h)$ and $r(h)$ denote the depth and cutrank of a PA derivation $h$, and define
\[g(\vec{n}) = max\{\vec{n},\vec{m},|\omega_{r(h_{\vec{n}})}(\omega \cdot d(h_{\vec{n}}))|\}\]
Fixing $\vec{n}$, Lemma \ref{complemma} yields $\mathbb{N}\models (\exists y\leq F_{\alpha}(max\{\underline{\vec{n}},\underline{\vec{m}}\})\exists\vec{z}E(\underline{\vec{n}},y,\vec{z},\underline{\vec{m}})$ with $\alpha=\omega_{r(h_{\vec{n}})}(\omega \cdot d(h_{\vec{n}}))$. Since $\alpha<\epsilon_0$, Lemma \ref{redprop} gives that $\epsilon_0\to_{|\alpha|}\alpha$ and so $\epsilon_0\to_{g(\vec{n})}\alpha$ since $g(\vec{n})\geq |\alpha|$. Thus $F_{\epsilon_0}(g(\vec{n}))\geq F_\alpha(g(\vec{n}))\geq F_\alpha(max\{\vec{n},\vec{m}\})\geq f(\vec{n})$.

Since $g$ is PA recursive, there is $\beta$ such that $g$ is dominated by $F_\beta(\max\{\vec{x}\})$ and so $f$ is dominated by $F_{\epsilon_0}(F_{\beta}(max\{\vec{x}\}))$ as required.
\end{proof}

\section{The Complexity of Selectors}
In the formulation of selector provability, we allowed the selector to be any provably recursive function in $S$. While primitive recursive functions seem more than sufficient in natural cases, so that in practice $I\Sigma_0+\text{Exp}$ or $I\Sigma_1$ suffice as a base theory, we now show that in principle the full proof-theoretic strength of $T$ is needed.\footnote{This is the case if we identify the proof-theoretic strength of $T$ with its supply of provably recursive functions.}

From Theorem \ref{class}, we can see that the provably recursive functions of PA can be classified by their position in the Wainer hierarchy. This leads to a corresponding classification of selector-provable schemes in terms of the minimal position of any selector. The result is the following.

\begin{theorem}\label{hier}
    For each $2\leq \alpha<\epsilon_0$, there is a $\Delta_1$ formula that is selector provable in PA by a selector belonging to $\mathcal{E}(F_\alpha)$, but not by any selector belonging to $\mathcal{E}(F_\beta)$ for $\beta<\alpha$.
\end{theorem}

The main step is the following.
\begin{lemma}\label{comlem}
    Let $f(x)$ be provably recursive in $T$. There is a $\Delta_1$ formula $G(x)$ selector provable in $T$ but for which any selector dominates $f$. Furthermore, if the graph of $f$ is $\Sigma_0$ definable then $G$ is selector provable by a function elementary recursive in $f$. 
\end{lemma}
Combining this with Lemma \ref{Sommer} and Theorem \ref{class} gives Theorem \ref{hier}.

\begin{proof}[Proof of Lemma \ref{comlem}]
We note consistency implies $\Pi_1$ soundness. Let $S= I\Sigma_0+\text{Exp}+f\halt$ so that $S\subseteq T$ and $S$ is sound. By diagonalisation, there is a formula $G(x)$ $\Delta_1$ in $S$ such that
\begin{equation}\label{G(x)}
    S\vdash G(x)\bimpl (\forall y\leq f(x))\neg (y\co_T G(x))
\end{equation}
 
First, note that $\forall x G(x)$ is true. If $G(\underline{n})$ were false, then by (\ref{G(x)}) there would be $k\leq f(n)$ such that $k\co_T G(\underline{n})$ witnessing this. But then $T\vdash G(\underline{n})$ contradicting the $\Pi_1$ soundness of $T$. It follows that if $g$ is a selector for $G(x)$ then $g$ must dominate $f$. Indeed, if $g(n)<f(n)$, then $S\vdash g(\underline{n})<f(\underline{n})$ as a true $\Sigma_1$ fact, and so $S\vdash \neg G(\underline{n})$ contradicting the soundness of $S$.

Working in $S$, if $G(x)$ holds, then $\Box_S G(x)$ by strong $\Sigma_1$ completeness, whereas if $\neg G(x)$, then (\ref{G(x)}) yields $(\exists y\leq f(x))(y\co_T G(x)$ and thus $\Box_S G(x)$. So $S\vdash \forall x \Box_SG(x)$ and $G(x)$ is selector provable in $S$. If $f$ satisfies the extra conditions in the hypothesis, then the resulting selector is elementary recursive in $f$. By effectively replacing instances of the $\text{Exp}$ and $f\halt$ axioms with proofs of them in $T$, the desired selector for $T$ in $\mathcal{E}(f)$ is obtained.
\end{proof}
The method of diagonalisation to obtain formulas with long proofs dates back to G\"odel in his Speed Up Theorem, the first published proof of which was given in \cite{BussSpeedup}. Similar results in terms of the number of symbols in a proof can be found in \cite{ParikhEFA}.

Theorem \ref{hier} was obtained for PA, but the phenomenon is much more general. Rathjen \cite{RathjenSSC} shows that the uniform $\Sigma_0$ definition of the graphs of $F_\alpha$ can be extended up to the Bachmann-Howard ordinal. It is clear that the analogous result can be obtained for various well-understood theories, such as the fragments $I\Sigma_n$, $ATR_0$ or $ID_1$.
 
Finally, we note that the schemes in the proof are constructed by self-reference and are rather artificial. It would be interesting to see if natural examples of this phenomenon could be found. In other words, are there "natural" formulas $F(\vec{x})$ such that $\text{PA}\vdash \forall \vec{x}\Box_{\text{PA}}F(\vec{x})$ but $I\Sigma_1\not\vdash \forall \vec{x}\Box_{\text{PA}}F(\vec{x})$?

\printbibliography

\end{document}